\documentclass[10pt]{amsart}
\usepackage{amssymb,MnSymbol}
\usepackage{amsthm,amsmath}

\title[Computing subsignatures of systems]{Computing subsignatures of systems with exchangeable component lifetimes}

\author{Jean-Luc Marichal}
\address{Mathematics Research Unit, FSTC, University of Luxembourg, 6, rue Coudenhove-Kalergi, L-1359 Luxembourg, Luxembourg}
\email{jean-luc.marichal[at]uni.lu }

\date{July 11, 2014}

\begin{document}

\theoremstyle{plain}
\newtheorem{theorem}{Theorem}
\newtheorem{lemma}[theorem]{Lemma}
\newtheorem{proposition}[theorem]{Proposition}
\newtheorem{corollary}[theorem]{Corollary}
\newtheorem{fact}[theorem]{Fact}
\newtheorem*{main}{Main Theorem}

\theoremstyle{definition}
\newtheorem{definition}[theorem]{Definition}
\newtheorem{example}{Example}
\newtheorem{algorithm}{Algorithm}

\theoremstyle{remark}
\newtheorem*{conjecture}{onjecture}
\newtheorem{remark}{Remark}
\newtheorem{claim}{Claim}

\newcommand{\N}{\mathbb{N}}
\newcommand{\R}{\mathbb{R}}
\newcommand{\Q}{\mathbb{Q}}
\newcommand{\Vspace}{\vspace{2ex}}
\newcommand{\bfx}{\mathbf{x}}
\newcommand{\bfy}{\mathbf{y}}
\newcommand{\bfz}{\mathbf{z}}
\newcommand{\bfh}{\mathbf{h}}
\newcommand{\bfe}{\mathbf{e}}
\newcommand{\os}{\mathrm{os}}
\newcommand{\dd}{\,\mathrm{d}}

\begin{abstract}
The subsignatures of a system with continuous and exchangeable component lifetimes form a class of indexes ranging from the Samaniego signature to the
Barlow-Proschan importance index. These indexes can be computed through explicit linear expressions involving the values of the structure
function of the system. We show how the subsignatures can be computed more efficiently from the reliability function of the system via
identifications of variables, differentiations, and integrations.
\end{abstract}



\maketitle

\section{Introduction}

Consider an $n$-component system $(C,\phi)$, where $C$ is the set $\{1,\ldots,n\}$ of its components and $\phi\colon\{0,1\}^n\to\{0,1\}$ is its structure
function which expresses the state of the system in terms of the states of its components. We assume that the system is semicoherent, which
means that $\phi$ is nondecreasing in each variable and satisfies the conditions $\phi(0,\ldots,0)=0$ and $\phi(1,\ldots,1)=1$. We also assume
that the components have continuous and exchangeable lifetimes $T_1,\ldots,T_n$.

The author \cite{MarMata} recently introduced the concept of subsignature of a system as follows. Let $M$ be a nonempty subset of the set $C$
of components and let $m=|M|$. The \emph{$M$-signature} of the system is the $m$-tuple $\mathbf{s}_M=(s_M^{(1)},\ldots,s_M^{(m)})$, where $s_M^{(k)}$ is the
probability that the $k$-th failure among the components in $M$ causes the system to fail. That is,
$$
s_M^{(k)} ~=~ \Pr(T_S=T_{k:M}),\qquad k\in\{1,\ldots,m\}\, ,
$$
where $T_S$ and $T_{k:M}$ denote, respectively, the lifetime of the system and the $k$-th smallest lifetime of the components in $M$, i.e., the $k$-th order statistic obtained by rearranging the
variables $T_i$ ($i\in M$) in ascending order of magnitude. A \emph{subsignature} of the system is an $M$-signature for some $M\subseteq C$.

When $M=C$ the $M$-signature reduces to the \emph{signature} $\mathbf{s}=(s_1,\ldots,s_n)$ of the system, a concept introduced in 1985 by Samaniego~\cite{Sam85} to compare different system designs and to easily compute the reliability of any system.\footnote{Actually, Samaniego~\cite{Sam85} proved that, when the component lifetimes are independent and identically distributed, the system reliability can always be expressed as the sum of the order statistics distributions weighted by the signature $\mathbf{s}$; this result was also established \cite{NavRyc07} in the more general case of exchangeable lifetimes} When $M$ is a singleton $\{j\}$ the $M$-signature reduces to the $1$-tuple $s_{\{j\}}^{(1)} = \Pr(T_S=T_j)$, which is the
\emph{Barlow-Proschan index} for component $j$, a concept introduced in 1975 by Barlow and Proschan~\cite{BarPro75} to measure the importance of
the components. Thus, the subsignatures define a class of $2^n-1$ indexes that range from the standard signature (when $M=C$) to the
Barlow-Proschan index (when $M$ is a singleton).

The $M$-signature of a system can be computed through any of the following explicit formulas (see \cite{MarMata}):\footnote{Here and throughout we identify Boolean vectors $\bfx\in\{0,1\}^n$ and subsets $A\subseteq\{1,\ldots,n\}$ by setting $x_i=1$ if and only if $i\in A$. We thus use the
same symbol to denote both a function $f\colon\{0,1\}^n\to\R$ and its corresponding set function $f\colon 2^{\{1,\ldots,n\}}\to\R$ interchangeably.}
\begin{eqnarray}
s_M^{(k)} &=& \sum_{\textstyle{A\subseteq C\atop |M\cap A|=m-k+1}}\, \frac{m-k+1}{n\,{n-1\choose |A|-1}}\,\phi(A)-\sum_{\textstyle{A\subseteq C\atop
|M\cap A|=m-k}}\, \frac{k}{n\,{n-1\choose |A|}}\,\phi(A)\, ,\label{eq:fgt1}\\
s_M^{(k)} &=& \sum_{j\in M}\,\sum_{\textstyle{A\subseteq C\setminus\{j\}\atop |M\setminus A|=k}}\,\frac{1}{n\,{n-1\choose
|A|}}\,\big(\phi(A\cup\{j\})-\phi(A)\big)\, .\label{eq:fgt2}
\end{eqnarray}

Equations~(\ref{eq:fgt1}) and (\ref{eq:fgt2}) show that, under the exchangeable assumption, the subsignatures do not depend on the distribution of the variables $T_1,\ldots,T_n$ but only on the
structure function. When $M=C$, formula (\ref{eq:fgt1}) reduces to Boland's formula \cite{Bol01}
$$
s_k ~=~ \sum_{\textstyle{A\subseteq C\atop |A|=n-k+1}}\frac{1}{{n\choose |A|}}\,\phi(A)-\sum_{\textstyle{A\subseteq C\atop
|A|=n-k}}\frac{1}{{n\choose |A|}}\,\phi(A)\, .
$$
When $M=\{j\}$, formula (\ref{eq:fgt2}) reduces to Shapley-Shubik's formula \cite{Sha53,ShaShu54}
\begin{equation}\label{eq:rt7}
I_{\mathrm{BP}}^{(j)} ~=~ \sum_{A\subseteq C\setminus\{j\}} \frac{1}{n\,{n-1\choose |A|}}\,\big(\phi(A\cup\{j\})-\phi(A)\big)\, .
\end{equation}
The computation of the subsignatures by means of Eqs.~(\ref{eq:fgt1}) and (\ref{eq:fgt2}) may be cumbersome and tedious for large systems since it
requires the evaluation of $\phi(A)$ for every $A\subseteq C$. To overcome this issue, in this paper we show how these indexes can be computed
from simple manipulations of the reliability function of the structure $\phi$ such as identifications of variables and differentiations.

Recall that the \emph{reliability function} of the structure $\phi$ is the multilinear function $h\colon
[0,1]^n\to\R$ defined by
\begin{equation}\label{eq:345hg}
h(\bfx) ~=~ h(x_1,\ldots,x_n) ~=~ \sum_{A\subseteq C}\phi(A)\, \prod_{i\in A}x_i\,\prod_{i\in C\setminus A}(1-x_i).
\end{equation}
When the component lifetimes are independent, this function expresses the reliability of the system in terms of the component reliabilities (see \cite[Chap.~2]{BarPro81} for a background on reliability functions and \cite[Section 3.2]{Ram90} for a more recent reference). It is easy to see that this function can always be put in the standard
multilinear form
\begin{equation}\label{eq:345hg2}
h(\bfx) ~=~ \sum_{A\subseteq C}d(A)\,\prod_{i\in A}x_i\, ,
\end{equation}
where the link between the coefficients $d(A)$ and the values $\phi(A)$ is given through the conversion formulas
$$
d(A) ~=~ \sum_{B\subseteq A}(-1)^{|A|-|B|}\, \phi(B)\quad\mbox{and}\quad \phi(A) ~=~ \sum_{B\subseteq A}d(B)\, .
$$

\begin{example}
The structure of a system consisting of two components connected in parallel is given by
$$
\phi(x_1,x_2) ~=~ \max(x_1,x_2) ~=~ x_1\amalg x_2 ~=~ x_1+x_2-x_1 x_2{\,},
$$
where $\amalg$ is the (associative) coproduct operation defined by $x\amalg y = 1-(1-x)(1-y)$. Considering only the multilinear expression of function $\phi$, one immediately obtains the corresponding reliability function $h(x_1,x_2) = x_1+x_2-x_1 x_2$.\qed
\end{example}

For any function $f$ of $n$ variables, we denote its diagonal section $f(x,\ldots,x)$ simply by $f(x)$. For instance, from Eqs.~(\ref{eq:345hg}) and (\ref{eq:345hg2}) we derive
$$
h(x) ~=~ \sum_{A\subseteq C}\phi(A)\, x^{|A|}\, (1-x)^{n-|A|} ~=~ \sum_{A\subseteq C}d(A)\, x^{|A|}\, .
$$

Owen \cite{Owe72} observed that the right-hand expression in Eq.~(\ref{eq:rt7}), which is the Barlow-Proschan index for component $j$, can be
computed by integrating over $[0,1]$ the diagonal section of the $j$-th partial derivative of $h$. That is,
\begin{equation}\label{eq:sd5f}
I_{\mathrm{BP}}^{(j)} ~=~ \int_0^1(\partial_jh)(t)\, dt\, .
\end{equation}
Thus, this formula provides a simple way to compute $I_{\mathrm{BP}}^{(j)}$ from the reliability function $h$  (at least simpler than the use of Eq.~(\ref{eq:rt7})). As a by-product, from Eq.~(\ref{eq:sd5f}) we easily derive the following integral formula
$$
\sum_{j=1}^n I_{\mathrm{BP}}^{(j)}\, x_j ~=~ \int_0^1\frac{d}{dt}{\,}h\big(t\, x_1+z(1-x_1),\ldots,t\, x_n+z(1-x_n)\big)\big|_{z=t}^{\mathstrut}\, dt\, .
$$

\begin{example}\label{ex:21j4hg}
Consider the bridge structure as indicated in Figure~\ref{fig:bs}. The corresponding structure and reliability functions are respectively given by
$$
\phi(x_1,\ldots,x_5) ~=~ x_1\, x_4\amalg x_2\, x_5\amalg x_1\, x_3\, x_5\amalg x_2\, x_3\, x_4
$$
and
\begin{eqnarray*}
h(x_1,\ldots,x_5) &=& x_1 x_4 + x_2 x_5 + x_1 x_3 x_5 + x_2 x_3 x_4 \\
&& \null - x_1 x_2 x_3 x_4 - x_1 x_2 x_3 x_5 - x_1 x_2 x_4 x_5  - x_1 x_3 x_4 x_5 - x_2 x_3 x_4 x_5\\
&& \null + 2\, x_1 x_2 x_3 x_4 x_5\, .
\end{eqnarray*}
By using Eq.~(\ref{eq:sd5f}) we immediately obtain
$(I_{\mathrm{BP}}^{(1)},\ldots,I_{\mathrm{BP}}^{(5)})=\big(\frac{7}{30},\frac{7}{30},\frac{1}{15},\frac{7}{30},\frac{7}{30}\big)$. Indeed, we have for instance
$$
I_{\mathrm{BP}}^{(3)} ~=~ \int_0^1(\partial_3h)(t)\, dt ~=~ \int_0^1 (2t^2-4t^3+2t^4)\, dt ~=~ \textstyle{\frac{1}{15}}{\,}.
$$
\qed
\end{example}

\setlength{\unitlength}{4ex}
\begin{figure}[htbp]\centering
\begin{picture}(11,4)
\put(3,0.5){\framebox(1,1){$2$}} \put(3,2.5){\framebox(1,1){$1$}} \put(5,1.5){\framebox(1,1){$3$}} \put(7,0.5){\framebox(1,1){$5$}}
\put(7,2.5){\framebox(1,1){$4$}}%
\put(0,2){\line(1,0){1.5}}\put(1.5,2){\line(2,-1){1.5}}\put(5.5,0){\line(-2,1){1.5}}\put(1.5,2){\line(2,1){1.5}}\put(5.5,4){\line(-2,-1){1.5}}%
\put(0,2){\circle*{0.15}}%
\put(9.5,2){\line(1,0){1.5}}\put(5.5,0){\line(2,1){1.5}}\put(9.5,2){\line(-2,-1){1.5}}\put(5.5,4){\line(2,-1){1.5}}\put(9.5,2){\line(-2,1){1.5}}%
\put(11,2){\circle*{0.15}}%
\put(5.5,0){\line(0,1){1.5}}\put(5.5,4){\line(0,-1){1.5}}
\end{picture}
\caption{Bridge structure} \label{fig:bs}
\end{figure}
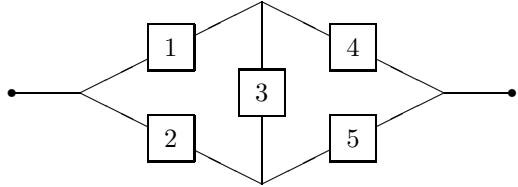

\begin{remark}
Example~\ref{ex:21j4hg} illustrates the fact that the reliability function $h$ can be easily obtained from the minimal path sets of the system simply by expanding the coproducts in $\phi$ and then simplifying the resulting polynomial expression (using $x_i^2=x_i$).
\end{remark}

Similarly to Owen's method, in this note we provide a simple way to compute the system subsignatures only from the reliability function $h(\bfx)$, thus avoiding formulas (\ref{eq:fgt1}) and (\ref{eq:fgt2}) which require the evaluation of $\phi(A)$ for every $A\subseteq C$.

When $M$ is a singleton, our method reduces to Owen's. When $M=C$, it reduces to the following algorithm (obtained in \cite{Mar014}) for the computation of the Samaniego
signature.

Let $f$ be a univariate polynomial of degree $\leqslant n$,
$$
f(x) ~=~ a_n\, x^n+\cdots + a_1\, x+ a_0\, .
$$
The $n$-\emph{reflected} of $f$ is the polynomial $R^n f$ defined by
$$
(R^n f)(x) ~=~ a_0\, x^n+a_1\, x^{n-1}+\cdots + a_n\, ,
$$
or equivalently, $(R^n f)(x)=x^n\, f(1/x)$.

\begin{algorithm}\label{algo:s}
The following algorithm inputs the number $n$ of components and the reliability function $h(x)$ and outputs the signature $\mathbf{s}$ of the
system.

\smallskip

\begin{quote}
\begin{enumerate}
\item[\textbf{Step 1.}] Let $g(x)=Dh(x)$ be the derivative of $h(x)$.

\item[\textbf{Step 2.}] For every $k\in\{1,\ldots,n\}$, let $c_{k-1}$ be the coefficient of $x^{k-1}$ in the $(n-1)$-degree polynomial $(R^{n-1}g)(x+1) = (x+1)^{n-1}\, g\big(\frac{1}{x+1}\big)$.

\item[\textbf{Step 3.}] We have $s_k = c_{k-1}/(k{\,}{n\choose k})$ for $k=1,\ldots,n$.
\end{enumerate}
\end{quote}

\smallskip
\end{algorithm}

Even though such an algorithm can be easily executed by hand for small $n$, a computer algebra system can be of great assistance for large $n$.

\begin{example}
Consider the bridge structure as indicated in Figure~\ref{fig:bs}. For this structure we have
$$
h(x) ~=~ 2x^2+2x^3-5x^4+2x^5,\qquad g(x) ~=~ 4x+6x^2-20x^3+10x^4\, ,
$$
and
$$
(R^4g)(x) ~=~ 10-20x+6x^2+4x^3.
$$
It follows that $(R^4g)(x+1)=4 x + 18 x^2 + 4 x^3$ and hence $\mathbf{s}=\big(0,\frac{1}{5},\frac{3}{5},\frac{1}{5},0\big)$. Indeed, we have for instance $s_3=c_2/(3{5\choose 3}) =\frac{3}{5}$.\qed
\end{example}

Denoting the coefficient of $x^{k-1}$ in the polynomial $f(x)$ by $[x^{k-1}]f(x)$, Algorithm~\ref{algo:s} can be summarized into the single equation (see \cite{Mar014})
$$
s_k ~=~ \frac{1}{k\,{n\choose k}}\, [x^{k-1}]\big((R^{n-1}Dh)(x+1)\big)\, ,\qquad k=1,\ldots,n\, .
$$

This note is organized as follows. In Section 2 we provide an algorithm which subsumes both Owen's method and Algorithm~\ref{algo:s} for the computation of the system subsignatures from the reliability
function. We also show how to compute generating functions of subsignatures. In Section 3 we discuss the concept of $M$-signature in the special case
where $M$ is a modular set of the system.

\section{An algorithm for the computation of subsignatures}

We now present our main result, namely an algorithm for the computation of the system subsignatures from the reliability function.

\begin{algorithm}\label{algo:ss}
The following algorithm inputs a subset $M$ of $m$ components and the reliability function $h(\bfx)$ and outputs the $M$-signature $\mathbf{s}_M$ of
the system.

\smallskip

\begin{quote}
\begin{enumerate}
\item[\textbf{Step 1.}] Let $h(x,z)$ be the bivariate polynomial obtained from the reliability function $h(\bfx)$ by identifying to $x$ the variables
in $M$ and identifying to $z$ the variables in $C\setminus M$.

\item[\textbf{Step 2.}] Let $g(x,z)=\partial_x\,h(x,z)$.

\item[\textbf{Step 3.}] For every $k\in\{1,\ldots,m\}$, let $c_{k-1}(z)$ be the coefficient of $x^{k-1}$ in the polynomial $(R_1^{m-1}g)(x+1,z) = (x+1)^{m-1}\, g\big(\frac{1}{x+1}{\,},z\big)$, where $R_1^{m-1}g$ is the $(m-1)$-reflected of $g$ with respect to its first argument.

\item[\textbf{Step 4.}] We have $s_M^{(k)} = \int_0^1t^{m-k}\, (1-t)^{k-1}\, c_{k-1}(t)\, dt$ for $k=1,\ldots,m$.
\end{enumerate}
\end{quote}

\smallskip
\end{algorithm}

\begin{proof}[Proof of Algorithm~\ref{algo:ss}]
By Eq.~(\ref{eq:345hg}) we have
$$
h(\bfx) ~=~ \sum_{A\subseteq C}\phi(A)\,\prod_{i\in M\cap A}x_i\,\prod_{i\in M\setminus A}(1-x_i)\,\prod_{i\in A\setminus M}x_i\,\prod_{i\in
C\setminus (A\cup M)}(1-x_i)\, .
$$
It follows that
\begin{eqnarray*}
h(x,z) &=& \sum_{A\subseteq C}\phi(A)\, x^{|M\cap A|}\,(1-x)^{|M\setminus A|}\, z^{|A\setminus M|}\, (1-z)^{n-|A|-|M\setminus A|}\, ,\\
g(x,z) &=& \sum_{A\subseteq C}\phi(A)\, |M\cap A|\, x^{|M\cap A|-1}\, (1-x)^{|M\setminus A|}\,  z^{|A\setminus M|}\, (1-z)^{n-|A|-|M\setminus A|}\\
&& \null -\sum_{A\subseteq C}\phi(A)\,|M\setminus A|\, x^{|M\cap A|}\, (1-x)^{|M\setminus A|-1}\,  z^{|A\setminus M|}\, (1-z)^{n-|A|-|M\setminus
A|}\, ,
\end{eqnarray*}
and
\begin{eqnarray*}
(R_1^{m-1}g)(x+1,z) &=& \sum_{A\subseteq C}\phi(A)\, |M\cap A|\, x^{|M\setminus A|}\, z^{|A\setminus M|}\, (1-z)^{n-|A|-|M\setminus A|}\\
&& \null -\sum_{A\subseteq C}\phi(A)\,|M\setminus A|\, x^{|M\setminus A|-1}\, z^{|A\setminus M|}\, (1-z)^{n-|A|-|M\setminus A|}\, .
\end{eqnarray*}
Let $k\in\{1,\ldots,m\}$. The coefficient of $x^{k-1}$ in the polynomial $(R_1^{m-1}g)(x+1,z)$ is then given by
\begin{eqnarray*}
c_{k-1}(z) &=& \sum_{\textstyle{A\subseteq C\atop |M\setminus A|=k-1}}\phi(A)\, |M\cap A|\, z^{|A\setminus M|}\, (1-z)^{n-|A|-|M\setminus A|}\\
&& \null -\sum_{\textstyle{A\subseteq C\atop |M\setminus A|=k}}\phi(A)\,|M\setminus A|\, z^{|A\setminus M|}\, (1-z)^{n-|A|-|M\setminus A|}\, .
\end{eqnarray*}
Thus we have
\begin{eqnarray*}
t^{m-k}\, (1-t)^{k-1}\, c_{k-1}(t) &=& \sum_{\textstyle{A\subseteq C\atop |M\cap A|=m-k+1}}\phi(A)\, (m-k+1)\, t^{|A|-1}\, (1-t)^{n-|A|}\\
&& \null -\sum_{\textstyle{A\subseteq C\atop |M\cap A|=m-k}}\phi(A)\, k\, t^{|A|}\, (1-t)^{n-|A|-1}\, .
\end{eqnarray*}
By integrating the right-hand side over $[0,1]$ and then using the classical identity
\begin{equation}\label{eq:ClasId}
\int_0^1 t^p\,(1-t)^q\, dt ~=~ \frac{1}{(p+q+1)\, {p+q\choose p}}\, ,\qquad p,q\in\N\, ,
\end{equation}
we precisely obtain the right-hand side of Eq.~(\ref{eq:fgt1}). This completes the proof.
\end{proof}

As mentioned in the introduction, the signature $\mathbf{s}$ of the system can be computed by Algorithm~\ref{algo:s}. Although this algorithm was established in \cite{Mar014}, we now show how it can be easily derived from Algorithm~\ref{algo:ss}.

\begin{proof}[Proof of Algorithm~\ref{algo:s}]
We only need to prove Step 3. Using Step 4 in Algorithm~\ref{algo:ss} and then Eq.~(\ref{eq:ClasId}), we see that $s_k = c_{k-1}\, \int_0^1t^{n-k}\, (1-t)^{k-1}\, dt = c_{k-1}/(k{\,}{n\choose k})$ for every $k=1,\ldots,n$, which is sufficient.
\end{proof}

Algorithm~\ref{algo:ss} shows that the $M$-signature can be computed from the bivariate polynomial $h(x,z)$ without the full knowledge of the reliability function $h(\bfx)$. Thus, two $n$-component systems having the same bivariate polynomial $h(x,z)$ also have the same $M$-signature.

The following proposition provides explicit expressions for the coefficient $c_{k-1}(z)$ in terms of the coefficients of $(x-1)^i$ and $x^i$ in
$g(x,z)$. The expression given in Eq.~(\ref{eq:7s6fd1}) is particularly interesting for small values of $k$, while that in Eq.~(\ref{eq:7s6fd2}) is interesting for small values of $m-k$. For instance, we obtain
$$
c_0(z) ~=~ g(1,z)\quad\mbox{and}\quad c_{m-1}(z) ~=~ g(0,z)\, .
$$

\begin{proposition}\label{prop:s7d5f}
Let $g(x,z)$ and $c_{k-1}(z)$ be the functions defined in Algorithm~\ref{algo:ss}. Then we have
\begin{eqnarray}
c_{k-1}(z) &=& \sum_{i=0}^{k-1}(-1)^i\, {m-1-i\choose m-k}\,\big([(x-1)^i]\, g(x,z)\big)\, ,\label{eq:7s6fd1}\\
c_{k-1}(z) &=& \sum_{i=0}^{m-k} {m-1-i\choose k-1}\,\big([x^i]\, g(x,z)\big)\, ,\label{eq:7s6fd2}
\end{eqnarray}
where $[(x-1)^i]\, g(x,z)=\frac{1}{i!}\,(\partial_1^ig)(1,z)$ and $[x^i]\, g(x,z)=\frac{1}{i!}\,(\partial_1^ig)(0,z)$.
\end{proposition}

\begin{proof}
We clearly have $g(x,z) = \sum_{i=0}^{m-1}\big([(x-1)^i]\, g(x,z)\big)\, (x-1)^i$ and hence
\begin{eqnarray*}
(R_1^{m-1}g)(x+1,z) &=& \sum_{i=0}^{m-1}\big([(x-1)^i]\, g(x,z)\big)\, (-1)^i\, (x+1)^{m-1-i}\\
&=& \sum_{i=0}^{m-1}\big([(x-1)^i]\, g(x,z)\big)\, (-1)^i\, \sum_{j=0}^{m-1-i}{m-1-i\choose j}\, x^{m-1-j}\\
&=& \sum_{j=0}^{m-1} x^{m-1-j}\,\sum_{i=0}^{m-1-j}(-1)^i\,{m-1-i\choose j}\,\big([(x-1)^i]\, g(x,z)\big)\, .
\end{eqnarray*}
Considering the coefficient of $x^{k-1}$ in the latter expression leads to formula~(\ref{eq:7s6fd1}). Formula~(\ref{eq:7s6fd2}) can be
established similarly.
\end{proof}

The following proposition gives an explicit expression for the \emph{generating function} $\sum_{k=1}^m s_M^{(k)}\, x^k$ of the $M$-signature in terms of the reliability function $h(\bfx)$.
Thus, it provides an alternative way to compute the $M$-signature.

\begin{proposition}\label{prop:sdf67}
Let $g(x,z)$ be the function defined in Algorithm~\ref{algo:ss}. Then we have
$$
\sum_{k=1}^m s_M^{(k)}\, x^k ~=~ \int_0^1 x\, R^{m-1}_t\,\big((R_1^{m-1}g)((t-1)\, x+1,\, z)\big)\big|_{z=t}^{\mathstrut}\, dt\, ,
$$
where $R_t^{m-1}$ denotes the $(m-1)$-reflection with respect to variable $t$. In particular,
$$
\sum_{k=1}^m s_M^{(k)} ~=~ \sum_{j\in M} I_{\mathrm{BP}}^{(j)} ~=~ \int_0^1 g(t,t)\, dt\, .
$$
\end{proposition}

\begin{proof}
By definition of the polynomial $R_1^{m-1}g$ in Algorithm~\ref{algo:ss} we have that
$$
\sum_{k=1}^m c_{k-1}(z)\, t^{k-1}\, x^{k-1} ~=~ (R_1^{m-1}g)(t\, x+1,z)\, .
$$
Multiplying through by $x$, replacing $t$ by $t-1$, and then applying $R^{m-1}_t$ to both sides, we obtain
$$
\displaystyle{\sum_{k=1}^m c_{k-1}(z)\, t^{m-k}\, (1-t)^{k-1}\, x^k} ~=~ x\, R^{m-1}_t\,\big((R_1^{m-1}g)((t-1)\, x+1,\, z)\big)\, .
$$
We then conclude by using Step 4 in Algorithm~\ref{algo:ss}. The particular case can be derived from the main result by setting $x=1$.
\end{proof}

From Proposition~\ref{prop:sdf67} we immediately derive the following algorithm for the computation of the generating function of the
$M$-signature. An advantage of this algorithm over Algorithm~\ref{algo:ss} is that it provides the $M$-signature without computing all the
coefficients $c_{k-1}(z)$.

\begin{algorithm}\label{algo:gfss}
The following algorithm inputs a subset $M$ of $m$ components and the reliability function $h(\bfx)$ and outputs the generating function of the
$M$-signature $\mathbf{s}_M$ of the system.

\smallskip

\begin{quote}
\begin{enumerate}
\item[\textbf{Step 1.}] Let $h(x,z)$ be the bivariate polynomial obtained from the reliability function $h(\bfx)$ by identifying to $x$ the variables
in $M$ and identifying to $z$ the variables in $C\setminus M$.

\item[\textbf{Step 2.}] Let $g(x,z)=\partial_x\,h(x,z)$.

\item[\textbf{Step 3.}] Let $f(t,x,z)=x\, (R_1^{m-1}g)((t-1)\, x+1,z)$.

\item[\textbf{Step 4.}] We have $\sum_{k=1}^m s_M^{(k)}\, x^k = \int_0^1 (R_1^{m-1}f)(t,x,t)\, dt$.
\end{enumerate}
\end{quote}

\smallskip
\end{algorithm}

\begin{example}
Let us consider again the bridge structure as indicated in Figure~\ref{fig:bs} and let us compute the generating function of the corresponding
$M$-signature for $M=\{1,2,3\}$. We have
\begin{eqnarray*}
h(x,z) &=& 2 x z + 2 x^2 z - 2 x^3 z - 3 x^2 z^2 + 2 x^3 z^2\, ,\\
g(x,z) &=& 2 z + 4 x z - 6 x^2 z - 6 x z^2 + 6 x^2 z^2\, ,\\
f(t,x,z) &=& -8 x^2 z + 8 t x^2 z + 2 x^3 z - 4 t x^3 z + 2 t^2 x^3 z + 6 x^2 z^2 - 6 t x^2 z^2\, ,
\end{eqnarray*}
and finally $\sum_{k=1}^3 s_M^{(k)}\, x^k = \frac{11}{30}\, x^2+\frac{1}{6}\, x^3$. Thus $\mathbf{s}_M=\textstyle{(0,\frac{11}{30},\frac{1}{6})}$.\qed
\end{example}

\section{Subsignatures associated with modular sets}

Suppose that the system contains a module $(M,\chi)$, where $M\subseteq C$ is the corresponding modular set
and $\chi\colon\{0,1\}^M\to\{0,1\}$ is the corresponding structure function. In this case the structure function of the system expresses through the composition
\begin{equation}\label{eq:8dsf6}
\phi(\bfx) ~=~ \psi\big(\chi(\bfx^M),\bfx^{C\setminus M}\big)\, ,
\end{equation}
where $\bfx^M=(x_i)_{i\in M}$ and $\bfx^{C\setminus M}=(x_i)_{i\in C\setminus M}$. The reduced system (of $n-m+1$ components) obtained from the original system $(C,\phi)$ by considering the modular set $M$ as a single macro-component $[M]$ will be denoted by $(C_M,\psi)$, where $C_M=(C\setminus M)\cup\{[M]\}$ and $\psi\colon\{0,1\}^{C_M}\to\{0,1\}$ is the organizing structure. For general background on modules, see \cite[Chap.~1]{BarPro81}.

Denote by $T_M$ the lifetime of the module and by $\mathbf{s}^M$ the signature of the module as an $m$-component system, that is, the
$m$-tuple whose $k$-th coordinate is given by $s_k^M = \Pr(T_M = T_{k:M})$ for $k=1,\ldots,m$.

The following theorem shows that $s_M^{(k)}$ factorizes into the product of $s_k^M$ and the expected
value of the function $(\partial_{[M]} h_{\psi})(t)$ with respect to a certain beta distribution, where $h_{\psi}\colon [0,1]^{C_M}\to\R$ is the reliability function of the structure $\psi$. When $M$ is a singleton, this result reduces to Owen's formula (\ref{eq:sd5f}).

This result was established in \cite[Cor.~16]{MarMata}. Here we give a simpler proof based on Algorithms~\ref{algo:s} and \ref{algo:ss}.

\begin{theorem}[{\cite{MarMata}}]\label{thm:76eqw}
For every nonempty modular set $M\subseteq C$ and every $k\in\{1,\ldots,m\}$, we have
$$
s_M^{(k)} ~=~ s_k^M\,\int_0^1 r_{k,m}(t)\, (\partial_{[M]} h_{\psi})(t)\, dt\, ,
$$
where $r_{k,m}(t)$ is the p.d.f.\ of the beta distribution on $[0,1]$ with
parameters $\alpha=m-k+1$ and $\beta=k$.
\end{theorem}

\begin{proof}
We prove the result by using Algorithm~\ref{algo:ss}. Let $h_{\chi}\colon [0,1]^M\to\R$ be the reliability function of $\chi$. By Eq.~(\ref{eq:8dsf6}) we then have $h(\bfx) = h_{\psi}\big(h_{\chi}(\bfx^M),\bfx^{C\setminus M}\big)$. Since $\partial_{[M]} h_{\psi}$ does not depend upon its $[M]$-variable, by the chain rule we have
$$
g(x,z) ~=~ \frac{d}{dx}\, h_{\psi}\big(h_{\chi}(x),z\big) ~=~ g_{\chi}(x)\, (\partial_{[M]} h_{\psi})(z)\, ,
$$
where $g_{\chi}(x)=Dh_{\chi}(x)$. Thus, for every $k\in\{1,\ldots,m\}$, we have
$$
c_{k-1}(z) ~=~ [x^{k-1}](R_1^{m-1}g)(x+1,z) ~=~ (\partial_{[M]} h_{\psi})(z)~ [x^{k-1}](R_1^{m-1}g_{\chi})(x+1)\, .
$$
But by Algorithm~\ref{algo:s}, we have $[x^{k-1}](R_1^{m-1}g_{\chi})(x+1) = k\, {m\choose k}\, s_k^M$. Therefore, by
Algorithm~\ref{algo:ss}, we finally obtain
$$
s_M^{(k)} ~=~  s_k^M\, k\, {m\choose k}\, \int_0^1 t^{m-k}\, (1-t)^{k-1}\, (\partial_{[M]} h_{\psi})(t)\, dt\, ,
$$
where $k\, {m\choose k}=1/\int_0^1 t^{m-k}\, (1-t)^{k-1}\, dt$ (use Eq.~(\ref{eq:ClasId})). This proves the theorem.
\end{proof}

\section*{Acknowledgments}

This research is supported by the internal research project F1R-MTH-PUL-12RDO2 of the University of Luxembourg.

\end{document}